\documentclass[10pt]{article}
\usepackage{amssymb,latexsym,amsmath,epsfig,amsthm,amsxtra,bm} 

\makeatletter

\renewcommand\section{\@startsection {section}{1}{\z@}
{-30pt \@plus -1ex \@minus -.2ex}
{2.3ex \@plus.2ex}
{\normalfont\normalsize\bfseries}}

\renewcommand\subsection{\@startsection{subsection}{2}{\z@}
{-3.25ex\@plus -1ex \@minus -.2ex}
{1.5ex \@plus .2ex}
{\normalfont\normalsize\bfseries}}

\renewcommand{\@seccntformat}[1]{\csname the#1\endcsname. }

\makeatother

\newtheorem{lemma}{Lemma}
\newtheorem{conjecture}{Conjecture}
\newtheorem{proposition}{Proposition}

\allowdisplaybreaks

\begin{document}

\begin{center}
\uppercase{\bf Statistical distribution of roots of a polynomial modulo primes}
\vskip 10pt
{\bf Yoshiyuki Kitaoka}\\

{\tt kitaoka@meijo-u.ac.jp}
\vskip 10pt
\end{center}
\vskip 10pt

%
\centerline{\bf Abstract}

\noindent
Let $f(x)=x^n+a_{n-1}x^{n-1}+\dots+a_0$ be an irreducible polynomial  with integer coefficients. For a prime $p$ for which $f(x)$ is fully splitting modulo $ p$, we consider $n$ roots $r_i$ of $f(x)\equiv 0\bmod p$  with  $0 \le r_1\le\dots\le r_n<p$ and propose several conjectures on the distribution of an integer $\lceil \sum_{i\in S} r_i/p\rceil$ for a subset $S$ of $\{1,\dots,n\}$ when $p\to\infty$.

\section{Introduction}\label{sec1}
Throughout this paper, unless otherwise specified, a polynomial means a monic {\it irreducible} polynomial of degree $>1$ with integer
coefficients, and the letter $p$ denotes a prime number. For a polynomial $f(x)=x^n+a_{n-1}x^{n-1}+\dots+a_0$ of degree $n$ and a prime number $p$, we say that $f(x)$ is {\it fully splitting modulo} $p$ if there are integers $r_1,r_2,\dots,r_n$ satisfying $f(x) \equiv \prod(x-r_i)\bmod p$. We assume that 
\begin{equation}\label{eq1}
 0\le r_1,\dots, r_n<p.
\end{equation}
Substituting
\begin{equation*}
Spl(f,X):=\left\{p\le X\mid f(x) \text{ is fully splitting modulo
}p\right\}
\end{equation*}
for a positive number $X$ and $Spl(f):=Spl(f,\infty),$ we know that $Spl(f)$ is an infinite set and that the density theorem due to Chebotarev holds; that is,
$$
\lim_{X\to\infty}\frac{\#Spl(f,X)}{\# \{p\le X\} }=\frac{1}{[\mathbb{Q}(f):\mathbb{Q}]},
$$
where $\mathbb{Q}$ is the rational number field and $\mathbb{Q}(f)$ is a finite Galois extension field of $\mathbb{Q}$ generated by all roots of $f(x)$ (\cite{Se}).

\noindent
For $p\in Spl(f)$, the definition of roots $r_i$ with \eqref{eq1} clearly implies that
\begin{equation}\label{eq2}
 a_{n-1}+r_1+\dots+r_n=C_p(f)p
\end{equation}
 for an integer $C_p(f)$. The author has previously studied the statistical distribution of $C_p(f)$ and local roots $r_i$ for $p\in Spl(f)$ (\cite{K1}--\cite{K3}, \cite{K4}, \cite{K5}). A basic fact that we need here is as follows.
\begin{proposition}\label{prop1}
 If $p\in Spl(f)$ is sufficiently large, then for any subset $S$ of $\{1,2,\dots,\linebreak n\}$ with  $\#S=n-1$, we have
\begin{equation}\label{eq3}
\left\lceil\scalebox{0.8}{$\displaystyle\sum$}_{j\in S}r_j/p\right\rceil=C_p(f),
\end{equation}
where  $\lceil x \rceil$ is an integer such that $x\le\lceil x \rceil<x+1$.
\end{proposition}
A proof of Proposition~\ref{prop1} is given in \cite{K2}, where  it is initially supposed that a sequence of $n!$ points $(r_{\sigma(1)}/p,\dots,r_{\sigma(n-1)}/p)$ for all permutations $\sigma\in S_n$ is uniformly distributed in $[0,1)^{n-1}$ when $p\to\infty$ if a polynomial $f(x)$ is indecomposable. However, this turns out to be false (counterexamples in the case of $n=6$ are given in \cite{K5} and in Section~4 here). Here, a polynomial $f(x)$ is called decomposable if there are polynomials $g(x)$ and $h(x)$ satisfying $f(x)=g(h(x))$ and  $1<\deg h<\deg f$, and indecomposable otherwise. In this paper, we give detailed observations in the case of $n\le6$. To do so, we introduce an ordering among roots  $r_i$ as follows: 
\begin{equation}\label{eq4}
 0\le r_1\le\dots\le r_n<p.  
\end{equation}
This determines roots $r_i$ uniquely. We note that $r_1=0$ implies $a_0\equiv 0\bmod p$ and \eqref{eq4} is equivalent to $0<r_1<\dots<r_n<p$ for a sufficiently large $p\in Spl(f)$ by the irreducibility of $f(x)$.

In Section~2, we recall observations related to the uniform distribution, and in Section~3, we introduce a new density and give observations in the case of $\deg f\le 5$, where the density is independent of a polynomial if it is irreducible and indecomposable. In Section~4, we give observations in the case of $\deg f=6$, where the density depends on each polynomial. In Section~5, we give some theoretical results to analyze the data, although it is too far to clarify the whole picture. The data presented in this paper were obtained using pari/gp.\footnote{ The PARI~Group, PARI/GP version {\tt 2.7.0}, Bordeaux, 2014, http://pari.math.u-bordeaux.fr/.}

\section{Uniform distribution}\label{sec2}
Let us recall a fundamental fact about uniform distribution.
\begin{lemma}\label{lem1}
 For a natural number $n$, the volume of a subset of the unit cube $[0,1)^{n}$ defined by $\{ (x_1,\dots,x_n)\in[0,1)^n\mid x_1+\dots+x_n\le x\}$ is given by
\begin{equation*}
 U_n(x):=\frac{1}{n!}\sum_{i=0}^n (-1)^i{n\choose i}\max(x-i,0)^n,
\end{equation*}
and for an integer $k$  with $1\le k \le n$, we have
\begin{equation}\label{eq5}
U_n(k)-U_n(k-1)=\frac{1}{n!}\sum_{i=0}^k (-1)^i{n+1\choose i}(k-i)^n.
\end{equation}
\end{lemma}
See \cite{Fe} for a proof of the first statement, from which identity \eqref{eq5} follows easily. We note that $A(n,k):=n!(U_n(k)-U_n(k-1)   )$ $(1\le k \le n)$ is called an Eulerian number and satisfies
$$
A(1,1)=1,\,A(n,k)=(n-k+1)A(n-1,k-1)+kA(n-1,k).
$$
Necessary values of $A(n,k)$ in this paper are
$$
\begin{array}{|c||c|c|c|c|c|c|}
\hline
n\setminus k&\,\,1\,\,&2&3&4&5\,\,\\
\hline
2&1& 1& & & \\
3&1& 4& 1& &  \\
4&1& 11& 11& 1&  \\
5&1& 26& 66& 26& 1 \\
\hline
\end{array}
$$
and we note that
\begin{equation}\label{eq6}
vol\left(\left\{(x_1,\dots,x_n)\in[0,1)^n\mid \left\lceil\scalebox{0.8}{$\displaystyle\sum$}x_i\right\rceil= k\right\}\right) =\frac{A(n,k)}{n!}.
\end{equation}

For a polynomial   $f(x)$ of degree $n$,
 \eqref{eq2} implies
$$
r_1/p+\dots+r_n/p=C_p(f)-a_{n-1}/p,
$$
whose left-hand side is close to an integer $C_p(f)$ when $p$ is large. Thus, the sequence of points $(r_1/p,\dots,r_n/p)$ is not uniformly distributed in the unit cube $[0,1)^n$ as $p\to\infty$. However, the sequence of $n!$ points $(r_{\sigma(1)}/p, . . . , r_{ \sigma(n-1)}/p) $ for all $\sigma\in S_n$ is expected to be uniformly distributed in $[0, 1)^{n-1}$ for the majority of polynomials. This is true without exception in the case of $n=2$ \cite{DFI}, \cite{T}. If the expectation is true, then the density of the distribution of the value $C_p(f)$ in \eqref{eq2} is given by Lemma\,\ref{lem1} as follows:
\begin{equation}\label{eq7}
 \lim_{X\to\infty}\frac{\#\{p\in Spl(f,X)\mid C_p(f)=k\}}
{\#Spl(f,X)}=\frac{A(n-1,k)}{(n-1)!},
\end{equation}
since Proposition\,\ref{prop1} implies 
\begin{equation}\label{eq8}
\#\{p\in Spl(f,X)\mid C_p(f)=k\}=\#\{p\in Spl(f,X)\mid \lceil\scalebox{0.8}{$\displaystyle\sum$}_{i\in S}r_{i}/p\rceil=k\}+O(1)
\end{equation}
for any subset $S$ of $\{1,\dots,n\}$ with $\#S=n-1$. Computer experiments support \eqref{eq7} well.

\noindent
Although we began our study with the distribution of $C_p(f)$, which originated from \cite{HKKN} and \cite{kinnen}, it is more interesting in view of \eqref{eq7} and \eqref{eq8} to study the distribution of the value $\lceil  (\sum_{i\in S}r_i)/p\rceil$ with the condition \eqref{eq4} on local roots $r_i$ for a given subset $S$ of $ \{1,\dots,n\}$. We provide some observations in the following sections. 
\section{New density}
We introduce here a new type of distribution. Statements on the density without proof hereinafter are conjectures based on numerical experiments.

Let $f(x)$ be a polynomial of degree $n$ and let $p$ be a prime in $Spl(f)$. We assume the global order \eqref{eq4} on local roots; that is, we number local roots $r_i$ of $f(x)$ modulo $p$  as follows:
\begin{equation*}
0\le r_1\le\dots\le r_n<p\quad(f(r_i)\equiv0\bmod p) .
\end{equation*}
As noted above, we have  $0<r_1<\dots<r_n<p$ if $p$ is sufficiently large. Let us consider a more general density than the left-hand side of \eqref{eq7}. For a subset $S$ of $\{1,2,\dots,n\}$, we define a frequency table $Pr(f,S,X)$ by 
\begin{gather}\nonumber
Pr(f,S,X):=[F_1,\dots,F_{s}],
\\
\intertext{where  $s:=\#S$ and}
\label{eq9}
F_k:=F_k(f,S,X)=\frac{\#\{p\in Spl(f,X)\mid \lceil\sum_{i\in S}r_{i}/p\rceil=k\}} 
{\#Spl(f,X)}.
\end{gather}
It is clear that the assumption $0\le r_i<p$ $(i=1,\dots,n)$ implies $F_k=0$ unless $0\le k \le s$. We see easily that $\lim_{X\to\infty}F_0(f,S,X)=0$ since primes contributing to the numerator of \eqref{eq9} divide the constant term $a_0$ of $f(x)$.

\noindent
Next, note that we may confine ourselves to the case $2\le s \le n-1$. Suppose that $F_k(f,S,X) \ne 0$ with $s=1$, say $S=\{i\}$; then, the equation $\lceil r_i/p\rceil =k$ implies $k=1$ for every sufficiently large $p$, which implies $\lim_{X\to\infty}F_1(f,S,X) =1$ and $\lim_{X\to\infty}F_k(f,S,X) =0\,(k\ne1)$. When $s=n$, that is, $S=\{1,\dots,n\}$, we have
$$
\lceil\scalebox{0.8}{$\displaystyle\sum$}_{i\in S}r_{i}/p\rceil=\lceil C_p(f)-a_{n-1}/p\rceil
=
\left\{
\begin{array}{ll}
 C_p(f)&\text{ if }a_{n-1}\ge 0,
 \\
 C_p(f)+1&\text{ if }a_{n-1}<0,
\end{array}
\right.
$$
and so this case is reduced to the case of $s=n-1$ by \eqref{eq8}, which has been previously studied \cite{K1}--\cite{K3}, \cite{K4}.

Assuming that $s=n-1$ and $f$ is indecomposable, we expect that in the case of $n\le5$, a sequence of $n!$ points $(r_{\sigma(1)}/p,\dots,r_{\sigma(n-1)}/p)$ $(\sigma\in S_n)$ is uniformly distributed as $p\to\infty$, which implies \eqref{eq7}. However, this is not the case if $n=6$, as we will see later.

We abbreviate as
$$
 Pr(f,S):=\lim_{X\to\infty}Pr(f,S,X)=\lim_{X\to\infty}[F_1(f,S,X),\dots,F_{s}(f,S,X)],
$$
assuming that the limit exists, something that the author has no data to refute. The first expectation is as follows.
\begin{conjecture}\label{expect1}
Suppose that $f(x)$ is not equal to  $g(h(x))$ for any quadratic polynomial $h(x)$. Then, for every $j$ with $1\le  j \le n$, we have
\begin{equation*}
 Pr(f,  S  )=[1,1]/2\text{ for } S=\{j,n+1-j\} ,
\end{equation*}
where $[1,1]/2$ means $[1/2,1/2]$ for simplicity; we adopt this notation hereinafter.
\end{conjecture}

\noindent
We checked the following polynomials.
Let $BP$ be a polynomial of degree $n=4$, 5, or 6 with coefficients
       equal to $0$ or $1$, and let $\alpha$ be one of its roots.
For a number $\beta = \sum_{i=1}^n c_i\alpha^{i-1}$ with $0\le c_i
       \le 2$, we take a polynomial $f$ of degree $n$ for which $\beta$ is a
       root.
We skip a reducible polynomial and also a decomposable polynomial, which is in the
form $f(x)=g(h(x))$ with $\deg h =2$. Considering that 
\begin{equation*}
F_k:=F_k(f,S,X) \to1/2
\quad(k=1,2,\,X\to\infty)
\end{equation*}
holds under Conjecture 1, we judge that the expectation is true if 
$$
|F_1-F_2|<0.1
$$
for a large number $X$, since $F_1+F_2 = 1$. 
\noindent
The excluded case is as follows.
\begin{proposition}
Suppose that a polynomial $f(x)=x^n+a_{n-1}x^{n-1}+\dots$ is equal to $g(h(x))$ for a quadratic polynomial $h(x)$. Then, for $S=\{j,n+1-j\}$ $(1\le  j\le n)$, we have
$$
Pr(f,S)=
\begin{cases}
 [1,0] & \text{ if } a_{n-1}\ge0,
\\
[0,1]& \text{ if } a_{n-1}<0.
\end{cases}
$$
\end{proposition}
\begin{proof}
We only have to see that, except for finitely many primes $p$, the value $\lceil(r_j+r_{n+1-j})/p\rceil$ is equal to $1$ or $2$ according to whether $a_{n-1}\ge0$ or $a_{n-1}<0$, respectively. We note that $\deg g=n/2$ and we may assume that $h(x)=(x+a)^2$ or $h(x)=(x+a)(x+a+1)$ for an integer $a$ according to whether the coefficient of $x$ of $h(x)$ is even or odd, respectively. In the case of $h(x)=(x+a)^2$, $a_{n-1}=an$ is easy, and if $r\in\mathbb Z$ $(0<r<p)$ is a root of $f(x)=g((x+a)^2)\equiv0\bmod p$, then $p-r-2a$ is also one of its roots, and we see $0<p-r-2a<p$ for a sufficiently large $p$ (\cite{K5}).Hence, by the assumption \eqref{eq4}, the sequences $r_1<\dots<r_n$ and $p-r_n-2a<\dots<p-r_1-2a$ are identical. Thus, we have $r_j+r_{n+1-j}=p-2a=p-2a_{n-1}/n$, which implies
$$
\begin{cases}
 (r_j+r_{n+1-j})/p\le 1 & \text{ if }a_{n-1}\ge0,
\\
(r_j+r_{n+1-j})/p>1 & \text{ if }a_{n-1}<0.
\end{cases}
$$
This completes the proof in the case of $h(x)=(x+a)^2$. In the case of $h(x)=(x+a)(x+a+1)$, noting that $a_{n-1}=(1+2a)n/2$ and both  $r_i$ and $p-r_i-1-2a\,(i=1,\dots,n)$ are roots, we have $r_j+r_{n+1-j}=p-1-2a=p-2a_{n-1}/n$ in a similar way as above, which completes the proof.
\end{proof}

For a subset $S$ of $\{1,\dots,n\}$, we put $S\spcheck:=\{n+1-i\mid i\in S\}$. Then, for $Pr(f,S)=[F_1,\dots,F_{s}]$, we have 
\begin{equation*}
 Pr(f,S\spcheck)=[F_{s},\dots,F_1]
\end{equation*}
empirically in many cases, which is equivalent to 
\begin{equation}\label{eq10}
 Pr(f,S\spcheck)=Pr(f,S)\spcheck,
\end{equation}
putting $[a_1,\dots,a_s]\spcheck:=[a_s,\dots,a_1]$.
\begin{proposition}\label{prop3}
Under the assumption that 
\begin{center}
 \textnormal{(A)} $\sum_{j\in S\spcheck}r_j/p$ is not an integer for every sufficiently large prime $p\in Spl(f)$,
\end{center}
we have $$Pr((-1)^nf(-x),S)\spcheck=Pr(f(x),S\spcheck).$$ Moreover, if $Pr((-1)^n f(-x),S)=Pr(f(x),S)$ holds, then we have \eqref{eq10}.
\end{proposition}
\begin{proof}
Since we have $f(x)\equiv\prod(x-r_i)\bmod p$ with $0<r_1<\dots< r_n<p$ for  a 
sufficiently large prime $p$,
we get $(-1)^nf(-x)\equiv\prod(x+r_i)\equiv\prod(x-R_i)\bmod p$ with
$$0<R_1:=p-r_n<\dots<R_i:=p-r_{n+1-i}<\dots< R_n:= p-r_1<p.$$
Noting an equality $\lceil s-r\rceil=s+1-\lceil  r\rceil$
for $s:=\#S\in\mathbb Z$, $r\not\in\mathbb Z$,
we see that 
$\lceil\sum_{i\in S}R_i/p\rceil=\lceil s-\sum_{j\in S\spcheck}r_j/p\rceil=s+1-\lceil\sum_{j\in
S\spcheck}r_j/p\rceil$,
which implies
$F_k((-1)^nf(-x),S,X)=F_{s+1-k}(f(x),S\spcheck,X)$.
Hence, we have the desired equation
 $Pr((-1)^n f(-x),S,X)=Pr(f(x),\linebreak[3]S\spcheck,X)\spcheck$.
\end{proof}
\noindent
{\bf Remark 1.}
If $f$ is indecomposable with $n\le5$, 
then $Pr(f,S)$ seems to be dependent on only $S$ and $\deg f$, as we see below.
Hence, this proposition elucidates \eqref{eq10}.
Therefore, The assumption \textnormal{(A)} is not necessarily true.
For example, for $f=x^4+1$, both $r_1<\dots<r_4$ and $p-r_4<\dots<p-r_1$
 are the set of local roots.
Hence, we have $r_1=p-r_4$ and $r_2=p-r_3$, that is, $\sum_{i\in S}r_i/p=1$ for $S=\{1,4\},\{2,3\}$.
Another example is the polynomial $f_3$ (cf. Remark 4).

Before giving a sufficient condition to (A),
let us recall a relation between the decomposition of a polynomial $f(x)$
 modulo $p$ and that of $p$ to the product of prime ideals over
 $F:=\mathbb{Q}(\alpha)$, where $\alpha$ is a root of $f(x)$.
Denote the ring of integers of $F$ by $O_F$ and prime ideals lying above $p$ by
 ${\mathfrak p}_i$.
Suppose that $p\in Spl(f)$ is sufficiently large and $r_1,\dots,r_n$ are roots of $f(x)$ modulo $p$\,; 
then, we have  the decomposition of $p$\,:  
$pO_F ={\mathfrak p}_1\cdots{\mathfrak p}_{n} $
and we may suppose that, by renumbering
\begin{equation}\label{eq11}
 {\mathfrak p}_i=(\alpha-{r}_i)O_F + pO_F
\text{ and } O_F/pO_F \cong O_F/{\mathfrak p}_1\oplus\dots\oplus O_F/{\mathfrak p}_{n},
\end{equation}
in particular $\alpha\equiv r_i\bmod {\mathfrak p}_i$.
The isomorphism in \eqref{eq11} is given by
\begin{equation*}
\beta \bmod pO_F \mapsto  (\beta \bmod {\mathfrak p}_1,\dots,\beta \bmod {\mathfrak p}_{n})
\end{equation*}
and
\begin{equation*}
 O_F/{\mathfrak p}_i \cong {\mathbb{Z}}/p{\mathbb{Z}}.
\end{equation*}
Moreover, $p$ splits fully over $F$ if and only if $p$ splits fully over the
field $\mathbb{Q}(f)$ generated by all roots of $f(x)$.
\begin{proposition}\label{prop4}
If  the condition \textnormal{(A)} for $S$ with $\#S=s$ does not hold,
then   a sum of  some $s$  roots of $f(x)$  is  zero,
that is, $f(x)=(x^{s}+0\cdot x^{s-1}+\dots)(x^{n-s}+a_{n-1}x^{n-s-1}+\dots)$.
\end{proposition}
\begin{proof}
The assumption means that there are infinitely many primes $p$ such that 
$\sum_{j\in S\spcheck} r_j \equiv 0\bmod p$.
Let $\alpha$ be a  root of $f(x)$ and put $F=\mathbb{Q}(\alpha)$, and let $K:=\mathbb{Q}(f)$
be a field generated by all roots of $f(x)$.
For a sufficiently large prime  $p\in Spl(f)$
and   roots $r_1,\dots,r_n$ of $f(x)$ modulo $p$ with \eqref{eq4},
let ${\mathfrak p}_i$ be a prime ideal of $F$ defined by \eqref{eq11}
and let ${\mathfrak p}_i={\mathfrak P}_{i,1}\dots{\mathfrak P}_{i,g}$ $(g=[K:F])$
be the decomposition of ${\mathfrak p}_i$ to the product of prime ideals over $K$.
The congruence $\alpha\equiv r_i\bmod {\mathfrak p}_i$ implies $\alpha\equiv r_i\bmod
{\mathfrak P}_{i,1}$.
Taking an automorphism $\sigma_i$ of $K$ over $\mathbb{Q}$ such that
${\mathfrak P}_{i,1}^{\sigma_i}={\mathfrak P}_{1,1}$,
we have   $\alpha^{\sigma_i}\equiv r_i\bmod {\mathfrak P}_{1,1}$.
Hence, $\sum_{i\in S\spcheck} \alpha^{\sigma_i} \equiv\sum_{i\in S\spcheck} r_i
\equiv 0\bmod {\mathfrak P}_{1,1}$ for infinitely many  prime numbers $p\in Spl(f)$.
Although automorphisms $\sigma_i$ depend on $p$,
we can choose an  appropriate infinite subset of $Spl(f)$ so that 
automorphisms $\sigma_i$ are independent of $p$. 
Hence, we have $\sum_{i\in S\spcheck} \alpha^{\sigma_i} \equiv\sum_{i\in S\spcheck} r_i
\equiv 0\bmod {\mathfrak P}_{1,1}$ for infinitely many primes $p$,
which implies  $\sum_{i\in S\spcheck} \alpha^{\sigma_i} =0$.
Since $\alpha^{\sigma_i}$ are distinct roots of $f(x)$ by $\alpha^{\sigma_i}\equiv r_i\bmod {\mathfrak P}_{1,1}$, we complete the proof. 
\end{proof}

Let us give some observations in the cases of $n=3,4,5$.
The case of $n=6$ is discussed in the following section.

In the case of  $n=3$, Conjecture 1 and \eqref{eq8} give
$$
Pr(f,S)=[1,1]/2 \text{ if } \#S=2.
$$

 In the case of $n=4$, supposing that $f$ is irreducible and indecomposable, We conjecture
\begin{align*}
&Pr(f,\{1,2\})=Pr(f,\{3,4\})\spcheck=[5,1]/6,
\\
&Pr(f,\{1,3\})=Pr(f,\{2,4\})\spcheck=[5,1]/6,
\\
&Pr(f,\{1,4\})=Pr(f,\{2,3\})=[1,1]/2,
\\
&Pr(f,S)=[1,4,1]/6\text{ if }s:=\#S=3.
\end{align*}
We note
\begin{equation*}
{n\choose s}^{-1}\sum_{\#S=s}Pr(f,S)=
\begin{cases}
 [1,1]/2!&\text{ if } s=2,
\\
[1,4,1]/3!&\text{ if } s=3,
\end{cases}
\end{equation*}
where ${n\choose s}$ is the number of subsets $S$ with $\#S=s$.
This  suggests that a sequence of points $[r_i,r_j]/p$ $(i\ne j)$
(resp. $[r_i,r_j,r_k]/p$ $(i\ne j, j\ne k,i\ne k)$) is uniformly
distributed in $[0,1)^2$ (resp. $[0,1)^3$) (cf.  \eqref{eq6}).
Thus, the symmetry \eqref{eq10} holds.
We checked the following polynomials.
Let $BP$ be an irreducible  polynomial of degree $4$ with  coefficients
       equal to $0$ or $1$, and let $\alpha$ be one of its roots.
For a number $\beta = \sum_{i=1}^4c_i\alpha^{i-1}$ with $0\le c_i \le 2$, 
we take a polynomial $f$ for which $\beta$ is a root,
but skip a reducible polynomial and a decomposable one.
We observe the behavior of values $6[F_1,F_2]-[5,1]$ for $S=\{1,2\}$ for increasing $X$, for example.
If the above conjecture is true, then it converges to $[0,0]$ when $X\to\infty$.
Defining an integer $X_j$ by $\#Spl(f,X_j)=1000j$,
we observe values   $|6F_1-5|+|6F_2-1|$ at $X=X_j$.
If they are less than $0.01$ for successive integers $X=X_j,\dots, X_{j+100}$ for
       some $j$,
we conclude that the above is true.

In the case of $n=5$\,
we adopt the following $d$-adic approximation method to find a candidate of the limit. First, we take the polynomial  $f=x^5-10x^3+5x^2+10x+1$, 
which defines a unique subfield of  degree $5$ in a cyclotomic field $\mathbb{Q}(\exp(2\pi i/25))$, and define an integer $X_j$ by $\#Spl(f,X_j)=1000j$ as before. Suppose that a sequence of vectors $c_m$ converges to a rational vector $\bm{a}=[a_1,\dots,a_s]/b$ $(a_i,b\in{\mathbb{Z}})$ and let $D$ be a finite set of integers including $b$. Then, for a large integer $m$, the error $\sum_i|dc_m[i] -r(dc_m[i])|$ is minimal at $d=b$, where $r(x)$ denotes the nearest integer to $x.$
Noting this, to guess the limit from a sequence $\{c_m\}$ given by computer experiments, we begin by guessing a set $D$ including the denominator $b$ of $\bm{a}$ by some means or other. In this case, we take for $D$ $\{d\mid 0<d\le500 \text{ and a prime divisor of $d$ is   $2,3$ or $5$}\}$. Second, we look for an integer $d=d_0\in D$ that gives the minimum of errors $\sum_i|dc_m[i] -r(dc_m[i])|$ $(d\in D)$. Then, $d_0$ is a candidate of the denominator. We checked that there is an integer $j$ such that for successive integers $X=X_j,\dots, X_{j+10^5}$, both the integer $d=d_0$ determined above and rounded integers of elements of $d\cdot Pr(f,S,X_i)$ are stable. In this case, the minimum error is less than $0.01 $ for $X=10^{10}$ and the conjecture holds.
The symmetry \eqref{eq10} holds.
\begin{align*}
& Pr(f,\{1, 2\})= Pr(f,\{4, 5\})\spcheck =[137, 7]/144,
\\
&Pr(f,\{1, 3\})=Pr(f,\{3, 5\})\spcheck=[11, 1]/12,
\\
&Pr(f,\{1, 4\})=Pr(f,\{2, 5\})\spcheck=[17, 7]/ 24,
\\
  &Pr(f,\{1, 5\})=[1, 1]/2,
\\
&Pr(f,\{2, 3\})=Pr(f,\{3, 4\})\spcheck=[29, 19]/48,
\\
&Pr(f,\{2,4\})=[1, 1]/2,
\\
&Pr(f,\{1,2,3\})=Pr(f,\{3,4,5\})\spcheck=[71, 67, 6]/144,
\\
&Pr(f,\{1,2,4\})=Pr(f,\{2,4,5\})\spcheck=[11, 12, 1]/24,
\\
&Pr(f,\{1,2,5\})=Pr(f,\{1,4,5\})\spcheck=[7, 39, 2]/48,
\\
&Pr(f,\{1,3,5\})=[1, 22, 1]/24,
\\
&Pr(f,\{2,3,4\})=[1, 7, 1]/9,
\\
&Pr(f,\{2,3,5\})=Pr(f,\{1,3,4\})\spcheck=[1, 17, 6]/24,
\\
&Pr(f,S)=[1, 11, 11, 1]/24\text{ if }\#S=4.
\end{align*}
We note that
\begin{equation*}
{n\choose s}^{-1} \sum_{\#S=s}Pr(f,S)=
\begin{cases}
 [1,1]/2! & \text{ if } s = 2,
\\
[1,4,1]/3!& \text{ if } s = 3,
\\
[1,11,11,1]/4!& \text{ if } s = 4.
\end{cases}
\end{equation*}

\noindent
To check other polynomials, we consider that an element $a$ of $Pr(f,S,X)$ converges to a candidate  $A/B$
$(A,B\in\mathbb Z)$
     if we have $A=r(a\cdot B)$.
By this method, we checked the above for any irreducible polynomial $f$ of degree $5$ that has a root $\sum_i c_i \alpha^{i-1}$ $(0\le c_i\le2)$, where $\alpha$ is a root of an irreducible polynomial with coefficients equal to $0$ or $1$.

\section{The case of degree 6}
In the case of $ n\le 5$, the classification of being decomposable or not
is enough to consider densities.
However, in the case of $n=6$, it is not enough and indecomposable
polynomials have been divided
into at least four types  so far
\footnote{
Added in the proof: This classification by densities turned out to be inappropriate.
}.
First, we give some examples.

\noindent
{\bf Example 1.} 
For the indecomposable polynomial 
$f=f_1=x^6+x^5+x^4+x^3+x^2+x+1$,
we expect
\begin{align*}
&Pr(f,\{1, 2\}) = Pr(f,\{5, 6\})\spcheck=[39, 1]/40,
\\
&Pr(f,\{1, 3\}) = Pr(f,\{4, 6\})\spcheck =[14, 1]/15,
\\
&Pr(f,\{1, 4\}) = Pr(f,\{3, 6\})\spcheck=[17, 3]/20,
\\
&Pr(f,\{1, 5\}) =Pr(f,\{2, 6\})\spcheck= [23, 7]/30,
\\
&Pr(f,\{1, 6\}) = [1, 1]/2, \\
&Pr(f,\{2, 3\}) = Pr(f,\{4, 5\})\spcheck =[19, 5]/24,
\\
&Pr(f,\{2, 4\}) = Pr(f,\{3, 5\})\spcheck=[3, 1]/4,
\\
&Pr(f,\{2, 5\}) = [1, 1]/2,\\
&Pr(f,\{3, 4\}) = [1, 1]/2, 
\end{align*}
\vspace{-8mm}
\begin{align*}
&Pr(f,\{1, 2, 3\}) =Pr(f,\{4, 5, 6\})\spcheck= [251, 106, 3]/360,
  \\
&Pr(f,\{1, 2, 4\}) =Pr(f,\{3, 5, 6\})\spcheck= [67, 52, 1]/120,\\
&Pr(f,\{1, 2, 5\}) =Pr(f,\{2, 5, 6\})\spcheck= [37, 82, 1]/120,    \\
&Pr(f,\{1, 2, 6\}) =Pr(f,\{1, 5, 6\})\spcheck = [16, 73, 1]/90, \\
&Pr(f,\{1, 3, 4\}) =  Pr(f,\{3, 4, 6\})\spcheck=[37, 82, 1]/120,   \\
&Pr(f,\{1, 3, 5\}) =  Pr(f,\{2, 4, 6\})\spcheck=[27, 92, 1]/120,    \\
&Pr(f,\{1, 3, 6\}) =  Pr(f,\{1, 4, 6\})\spcheck = [13, 104, 3]/120, \\
&Pr(f,\{1, 4, 5\}) =  Pr(f,\{2, 3, 6\})\spcheck =[11, 46, 3]/60,\\
&Pr(f,\{2, 3, 4\}) =   Pr(f,\{3, 4, 5\})\spcheck =[17, 50, 5]/72,  \\
&Pr(f,\{2, 3, 5\}) =  Pr(f,\{2, 4, 5\})\spcheck =[5, 16, 3]/24, 
\end{align*}
\vspace{-8mm}
\begin{align*}
&Pr(f,\{1, 2, 3, 4\})=Pr(f,\{3, 4, 5,6\})\spcheck = [25, 68, 26, 1]/120,       \\
&Pr(f,\{1, 2, 3, 5\}) = Pr(f,\{2, 4, 5, 6\})\spcheck = [20, 73, 26, 1]/120,       \\
&Pr(f,\{1, 2, 3, 6\}) =  Pr(f,\{1, 4, 5, 6\})\spcheck =[7, 178, 53, 2]/240,      \\
&Pr(f,\{1, 2, 4, 5\}) =     Pr(f,\{2, 3, 5, 6\})\spcheck = [10, 83, 26, 1]/120,  \\
&Pr(f,\{1, 2, 4, 6\}) =    Pr(f,\{1, 3, 5, 6\})\spcheck =[1, 89, 29, 1]/120,      \\
&Pr(f,\{1, 2, 5, 6\}) = [1, 59, 59, 1]/120, \\
&Pr(f,\{1, 3, 4, 5\}) =   Pr(f,\{2, 3, 4, 6\})\spcheck=[5, 83, 31, 1]/120,       \\
&Pr(f,\{1, 3, 4, 6\}) = [1, 59, 59, 1]/120, \\
&Pr(f,\{2, 3, 4, 5\}) = [1, 23, 23, 1]/48, \\
&Pr(f,S)= [1, 26, 66, 26, 1]/120\text{ if }\#S=5. 
\end{align*}

\noindent
{\bf Remark 2.}
In this case, the symmetry \eqref{eq10} holds.
To look for conjectural values, we adopt  the $10$-adic approximation besides the $d$-adic one
in the previous section.
That is, we observe the minimum of errors $\sum_i|c_m[i] -r(d\cdot c_m[i])/d|$ and
 $\sum_i|d\cdot c_m[i] -r(d\cdot c_m[i])|$ $(d\in D)$ for a sequence of vectors $c_m$.
In this case, we take for $D$ $\{d\mid 1\le d \le 500 $  and a prime divisor of $d$
is  $2,3$ or $5\}$.
Let $p_j$ be  the smallest prime number in $Spl(f)$ larger than $10^9j$.
To the extent of $p_j<10^{11}$ and $j>30$, the values of $Pr(f,S,p_j)$ support the above conjecture by this double-checking method.

We have
\begin{equation*}
{n\choose s}^{-1} \sum_{\#S=s}Pr(f,S)=
\begin{cases}
[1,1]/2! &\text{ if }s=2,
\\
[1,4,1]/3!& \text{ if }s=3,
\\
[1,11,11,1]/4!& \text{ if }s=4,
\\
[1,26,66,26,1]/5!& \text{ if }s=5,
\end{cases}
\end{equation*}
and
\begin{align}\label{eq12}
Pr(f,\{1,2,5\})&= Pr(f,\{1,3,4\}), \\\label{eq13}
Pr(f,\{2,5,6\})&= Pr(f,\{3,4,6\}),\\\label{eq14}
Pr(f,\{2,3,4\})+ Pr(f,\{1,5,6\})&= Pr(f,\{3,4,5\})+ Pr(f,\{1,2,6\}), \\
 \nonumber
Pr(f,\{2,3,5\})+ Pr(f,\{1,4,6\})&= Pr(f,\{1,3,5\})+ Pr(f,\{2,4,6\})=\\\label{eq15}
Pr(f,\{2,3,6\})+ Pr(f,\{1,4,5\})&=
Pr(f,\{2,4,5\})+ Pr(f,\{1,3,6\}),
\\\label{eq16}
Pr(f,\{1,2,5,6\})&=Pr(f,\{1,3,4,6\}).
\end{align}
$\mathbb{Q}(f)=\mathbb{Q}(\exp(2\pi i/7))$ is obvious.
The sequence $[1,59,59,1]$ for $S=\{1,2,5,\linebreak6\}$ and $\{1,3,4,6\}$ is given 
by $T_1(4,i)$ $(i=1,2,3,4)$, 
where $T_1(n,k)$ $(1\le k \le n)$ is defined by
$$
T_1(1,1)= 1,T_1(n,k)=(4n-4k+1)T_1(n-1,k-1)+(4k-3)T_1(n-1,k),
$$
and the sequence $[1,23,23,1]$ for $S=\{2,3,4,5\}$ is given by 
$T_2(4,i)\,(i=1,2,3,4)$ where $T_2(n,k)$ $(1\le k \le n)$ is defined by
$$
T_2(1,1)= 1,T_2(n,k)=(2n-2k+1)T_2(n-1,k-1)+(2k-1)T_2(n-1,k).
$$

\noindent
{\bf Example 2.} 
For the indecomposable polynomial 
$f=f_2=x^6-2x^5+11x^4+6x^3+16x^2+122x+127$, 
we expect 
\begin{align*}
& Pr(f,\{1, 2\}) = Pr(f,\{5, 6\})\spcheck = [139, 5]/144  , 
\\
&Pr(f,\{1, 3\}) =Pr(f,\{4, 6\}) \spcheck= [127, 17]/144,
\\
&Pr(f,\{1, 4\}) =Pr(f,\{3, 6\})\spcheck= [7, 2]/9,
\\
&Pr(f,\{1, 5\}) =Pr(f,\{2, 6\})\spcheck= [25, 11]/36,
\\
&Pr(f,\{1, 6\}) = [1, 1]/2,
\\
&Pr(f,\{2, 3\}) = Pr(f,\{4, 5\})\spcheck=[3, 1]/4 ,
\\
&Pr(f,\{2, 4\}) = Pr(f,\{3, 5\})\spcheck=[107, 37]/144,
\\
&Pr(f,\{2, 5\}) = [1, 1]/2,
\\
&Pr(f,\{3, 4\}) = [1, 1]/2,
\end{align*}
\vspace{-8mm}
\begin{alignat*}{2}
&Pr(f,\{1, 2, 3\}) = [49, 23, 0]/72  ,&&  Pr(f,\{4, 5, 6\}) = [0, 1, 3]/4   ,\quad(*)
\\
&Pr(f,\{1, 2, 4\}) = [37, 35, 0]/72  ,&& Pr(f,\{3, 5, 6\}) = [0, 23, 49]/72,\quad(*)
\\
&Pr(f,\{1, 2, 5\}) = [5, 13, 0]/18  , && Pr(f,\{2, 5, 6\}) = [0, 89, 55]/144     ,\quad(*)
\\
&Pr(f,\{1, 2, 6\}) = [28, 115, 1]/144 ,&\hspace{1mm}&  Pr(f,\{1, 5, 6\}) = [0, 13, 5]/18  ,\quad(*)
\\
&  Pr(f,\{1, 3, 4\}) = [5, 13, 0]/18,   &&Pr(f,\{3, 4, 6\}) = [0, 89, 55]/144,      \quad(*)
\\
&Pr(f,\{1, 3, 5\}) = [1, 3, 0]/4  ,    &&          Pr(f,\{2, 4, 6\}) =Pr(f,\{1, 3, 5\})\spcheck,
\\
&Pr(f,\{1, 3, 6\}) = [16, 121, 7]/144,&&Pr(f,\{1, 4, 6\}) = [1, 115, 28]/144,	   \quad(*)
\\
&Pr(f,\{1, 4, 5\}) = [3, 12, 1]/16 ,    &&           Pr(f,\{2, 3, 6\}) = [0, 3, 1]/4,\quad(*)
\\
&Pr(f,\{2, 3, 4\}) = [35, 101, 8]/144, &&Pr(f,\{3, 4, 5\}) = [0, 13, 5]/18  ,\quad(*)
\\
&Pr(f,\{2, 3, 5\}) = [29, 95, 20]/144   , &&Pr(f,\{2, 4, 5\}) = [8, 101, 35]/144,\quad(*)
\end{alignat*}
\begin{align*}
&Pr(f,\{1, 2, 3, 4\}) =Pr(f,\{3, 4, 5, 6\}) \spcheck= [31, 77, 36, 0]/144  , 
\\
&Pr(f,\{1, 2, 3, 5\}) =Pr(f,\{2, 4, 5, 6\})\spcheck=[19, 89, 36, 0]/144  ,
\\
&Pr(f,\{1, 2, 3, 6\}) =Pr(f,\{1, 4, 5, 6\})\spcheck= [0, 3, 1, 0]/4  ,   
\\
&Pr(f,\{1, 2, 4, 5\}) =Pr(f,\{2, 3, 5, 6\})\spcheck= [1, 26, 9, 0]/36  ,
\\
&Pr(f,\{1, 2, 4, 6\}) = Pr(f,\{1, 3, 5, 6\})\spcheck=[0, 107, 37, 0]/144   ,
\\
&Pr(f,\{1, 2, 5, 6\}) = [0, 1, 1, 0]/2,
\\
&Pr(f,\{1, 3, 4, 5\}) = Pr(f,\{2, 3, 4, 6\})\spcheck=[0, 25, 11, 0]/36 ,
\\
&Pr(f,\{1, 3, 4, 6\}) = [0, 1, 1, 0]/2 ,
\\
&Pr(f,\{2, 3, 4, 5\}) = [0, 1, 1, 0]/2,
\\
&Pr(f,S) = [0, 1, 2, 1, 0]/4\text{ if }\#S=5.
\end{align*}

\noindent
{\bf Remark 3.}
Conjectural values are determined by the double-checking method above.
In this case, the symmetry does not hold for lines with tag $(*)$ for $\#S=3$.
In this case, we have
\begin{equation*}
{n\choose s}^{-1} \sum_{\#S=s}Pr(f,S)=
\begin{cases}
 [1,1]/2 &\text{ if } s = 2,
\\
[3,13,4]/20 &\text{ if } s = 3,
\\
 [1,19,19,1]/40&\text{ if } s = 4,
\\
[0,1,2,1,0]/4&\text{ if } s = 5,
\end{cases}
\end{equation*}
and \eqref{eq12}, \eqref{eq13}, and \eqref{eq16} hold.
Putting $t(n,m)=2A(n+1,m+1)-{n\choose m}$,
we see that $[1,19,19,1]=[t(3,0),t(3,1),\linebreak[3] t(3,2),t(3,3)]$.
The polynomial $x^2/4+2x^3/4+x^4/4$ corresponding to $[0,1,2,1,0]/4$ for
 $\#S=5$ above is equal to $(x/2+x^2/2)^2$. That
 is, the generating polynomial of $Pr(f,S)$ is identical to the square of
 the generating polynomial of densities of the two-dimensional uniform
 distribution (cf.\ Section~4).
 This shows that a sequence of points $(r_{\sigma(1)},\dots,r_{\sigma(5)})/p$ for $\sigma\in S_6$ is 
 not uniformly distributed in $[0,1)^5$.

\noindent
Let $\alpha$ be a root of $f$.
Then, we have $\mathbb{Q}(\alpha)=\mathbb{Q}(f)=\mathbb{Q}(\exp(2\pi i/7))$,
and  over a quadratic subfield $M_2=\mathbb{Q}(\sqrt{-7})$  
of $\mathbb{Q}(\alpha)$, $f$
 has a divisor   $g_3(x):=x^3 - x^2 +(\sqrt{-7}+5) x +3\sqrt{-7}+8$,
for which the coefficient of $x^2$ is the rational number $-1$.
This is an example of the first case of Proposition\,\ref{prop5} below.

The densities for the polynomial $f_2(-x)$ are the same as those for the next polynomial
 $f_3(x)$; that is,
\begin{equation}\label{eq17}
 Pr(f_2(-x),S)=Pr(f_3(x),S).
\end{equation}

\noindent
{\bf Example 3.} 
For the indecomposable polynomial 
$f = f_3=x^6-2x^3+9x^2+6x+2$, 
we expect that
\begin{equation}\label{eq18}
Pr(f_3,S)=Pr(f_2,S\spcheck)\spcheck. 
\end{equation}

\noindent
{\bf Remark 4.}
Conjectural values are determined by the double-checking method.
Let us make a remark from a theoretical viewpoint.
Because we can check that the polynomial
 $f_2(x)$ satisfies the assumption \textnormal{(A)} by using Proposition \ref{prop4},
we have $Pr(f_2(-x),S)\spcheck =Pr(f_2(x),S\spcheck)$,
and hence, \eqref{eq17} and \eqref{eq18} are equivalent.
Polynomials $f_3(x)$ and $f_3(-x)$ have the same densities, and assumption
 \textnormal{(A)} on $S$ is not satisfied for
 either polynomial
if  $\#S=3$ and $S\ne\{1,3,5\},\{2,4,6\}$.

\noindent
Let $\alpha$ be a root of $f$. Then,
$\mathbb{Q}(\alpha)$  is a splitting field of the polynomial $x^3-3x-14$, 
which is the composite field of $\mathbb{Q}(\sqrt{-1})$ and a field defined by $x^3-3x-14=0$. 
Over a quadratic subfield $M_2=\mathbb Q(\sqrt{-1})$  of  $\mathbb{Q}(\alpha)$, 
$f$  has a divisor   $g_3(x):=x^3 +0\cdot x^2-3\sqrt{-1}x-\sqrt{-1}-1 $,
whose second leading coefficient is a rational number $0$.
This is also an example of the first case of Proposition\,\ref{prop5}.

\noindent
{\bf Example 4.} 
For an indecomposable polynomial 
$f =
f_4=x^6-9x^5-3x^4+139x^3+93x^2-627x+1289  $, we expect

\begin{align*}
& Pr(f,\{1, 2\}) =  Pr(f,\{5, 6\})\spcheck =[277, 11]/288,
\\
&Pr(f,\{1, 3\}) =  Pr(f,\{4, 6\})\spcheck =   [661, 59]/720,
\\
&Pr(f,\{1, 4\}) =Pr(f,\{3, 6\})\spcheck= [38, 7]/45,
\\
&Pr(f,\{1, 5\}) =  Pr(f,\{2, 6\})\spcheck= [559, 161]/720 ,
\\
&Pr(f,\{1, 6\}) = [1, 1]/2,
\\
&Pr(f,\{2, 3\}) =  Pr(f,\{4, 5\})\spcheck =[33, 7]/40,
\\
&Pr(f,\{2, 4\}) =  Pr(f,\{3, 5\})\spcheck =[47, 13]/60,  
\\
&Pr(f,\{2, 5\}) = [1, 1]/2,
\\
&Pr(f,\{3, 4\}) = [1, 1]/2,
\end{align*}
\begin{align*}
&Pr(f,\{1, 2, 3\}) = Pr(f,\{4, 5, 6\})\spcheck = [475, 164, 9]/648 ,
 \\
&Pr(f,\{1, 2, 4\}) = Pr(f,\{3, 5, 6\})\spcheck = [649, 416, 15]/1080,
\\
&Pr(f,\{1, 2, 5\}) =Pr(f,\{2, 5, 6\})\spcheck = [314, 751, 15]/1080,
\\
&Pr(f,\{1, 2, 6\}) = Pr(f,\{1, 5, 6\})\spcheck = [208, 857, 15]/1080 ,
\\
&Pr(f,\{1, 3, 4\}) = Pr(f,\{3, 4, 6\})\spcheck =  [314, 751, 15]/1080  ,
\\
&Pr(f,\{1, 3, 5\}) = Pr(f,\{2, 4, 6\})\spcheck= [15, 56, 1]/72          ,   
\\
&Pr(f,\{1, 3, 6\}) =  Pr(f,\{1, 4, 6\})\spcheck = [433, 2726, 81]/3240,
\\
&Pr(f,\{1, 4, 5\}) =  Pr(f,\{2, 3, 6\})\spcheck = [539, 2520, 181]/3240,
\\
&Pr(f,\{2, 3, 4\}) = Pr(f,\{3, 4, 5\})\spcheck =  [722, 2375, 143]/3240,
\\
&Pr(f,\{2, 3, 5\}) =  Pr(f,\{2, 4, 5\})\spcheck  = [639, 2314, 287]/3240.
\end{align*}
\begin{align*}
&Pr(f,\{1, 2, 3, 4\}) =  Pr(f,\{3, 4, 5, 6\})\spcheck=[53, 175, 56, 4]/288,  \\
&Pr(f,\{1, 2, 3, 5\}) =Pr(f,\{2, 4, 5, 6\})\spcheck = [101,469,140,10]/720 , \\
&Pr(f,\{1, 2, 3, 6\}) =  Pr(f,\{1, 4, 5, 6\})\spcheck=  [17, 268, 70, 5]/360 , \\
&Pr(f,\{1, 2, 4, 5\}) = Pr(f,\{2, 3, 5, 6\})\spcheck= [24, 261, 70, 5]/360 ,\\
&Pr(f,\{1, 2, 4, 6\}) =  Pr(f,\{1, 3, 5, 6\})\spcheck =  [5, 277, 73, 5]/360   , \\
&Pr(f,\{1, 2, 5, 6\}) = [1, 35, 35, 1]/72,\\
&Pr(f,\{1, 3, 4, 5\}) = Pr(f,\{2, 3, 4, 6\})\spcheck =  [27, 515, 168, 10]/720, \\
&Pr(f,\{1, 3, 4, 6\}) = [1, 35, 35, 1]/72,
\\
&Pr(f,\{2, 3, 4, 5\}) = [7, 137, 137, 7]/288,
\\
&Pr(f,S) =  [1, 14, 42, 14, 1]/72\text{ if }\#S=5.
\end{align*}

\noindent
{\bf Remark 5.}
The conjectural values above were determined by the double-checking
method  for $p <10^{13}$ and $D = \{d \mid d \le 4000$ 
and a prime divisor of $d$ is  $2,3$ or $5\}$.
The symmetry \eqref{eq10} and 
\eqref{eq12}--\eqref{eq16} hold.

\noindent
In this case, we expect
\begin{equation*}
{n\choose s}^{-1} \sum_{\#S=s}Pr(f,S)= 
\begin{cases}
 [1,1]/2!&\text{ if } s= 2,
\\
[1,4,1]/3!&\text{ if } s= 3,
\\
[1,11,11,1]/4! &\text{ if } s= 4,
\\
[1,14,42,14,1]/72 &\text{ if } s= 5.
\end{cases}
\end{equation*}
Substituting $T(m,n):= (mn)!\prod_{k=0}^{n-1}(k!/(k+m)!)$ $(m,n\ge1)$ $\textnormal{( \cite{Sl})}$,
 which is called a  multidimensional Catalan number,
we see $T(m,6-m)=1,14,42,14,1$ according to $m = 1,2,3,4,5$, respectively.

\noindent
Let $\alpha$ be a root of $f$.
Then, over a cubic subfield $M_3$ defined by $\beta^3 - \beta^2 -2\beta+1=0$ of 
$\mathbb{Q}(\alpha)$, $f$
 has a divisor   $g_2(x):=x^2 +( 6\beta-5) x+9\beta^2-15\beta+8$,
whose discriminant is the rational number $-7$.
This is an example of the second case of Proposition\,\ref{prop5}.

\noindent
We have $\mathbb{Q}(\alpha) = \mathbb{Q}(\exp(2\pi i/7))$ and
 $Pr(f_4(-x),S)=Pr(f_4(x), S)$.

\noindent
{\bf Remark 6.}
With respect to the polynomials on pp.\ 86--87 in \cite{K5}, the densities defined here for the polynomials in cases (1)--(5) are equal to the ones given in Examples 2, 3, 1, 1, and 4, respectively.

We can consider a more general  density.
For a real function $t=t(x_1,\dots,x_n)$, 
we define $Pr(f,t,X):=[\dots,F_0,F_1,\dots]$ by
$$
F_k:=\frac{\#\{p\in Spl(f,X)\mid \lceil t(r_1/p,\dots,r_n/p)\rceil=k\}} 
{\#Spl(f,X)}
$$
and put $Pr(f,t):=\lim_{X\to\infty}Pr(f,t,X)$.

\noindent
For example, for $f = x^3-3x+1$, we expect
$$
Pr(f,4x_i)_{[1..4]}=
\begin{cases}
 [9,5,2,0]/16&(i=1),
\\
[3,5,5,3]/16&(i=2),
\\
[0,2,5,9]/16&(i=3),
\end{cases}
$$
where $v_{[n..m]}$ means a subsequence $[v_n,\dots,v_m]$ for
$v=[\dots,v_0,v_1,\dots]$.

\noindent
$Pr(f,4x_1)[4]=Pr(f,4x_3)[1]=0$ is not difficult to prove.

\section{Arithmetic aspects}
We recall that in this paper, a polynomial is supposed to be an irreducible monic one with integer coefficients, and hereinafter, we neglect the global order \eqref{eq4}. To analyze the case of  $\deg f=6$, we prepare the following.
\begin{proposition}\label{prop5}
Let  $f(x)=x^{2m}+a_{2m-1}x^{2m-1}+\dots$ be a polynomial  of even degree $2m$
and let $\alpha$ be a root of $f(x)$ and put $F= \mathbb{Q}(\alpha)$.
Let $p$  be a sufficiently large prime number in $ Spl(f)$, 
and let $r_1,\dots,r_{2m}\in\mathbb{Z}$ be roots of $f(x)$ modulo $p$, that is,
\begin{equation}\label{eq19}
f(x) \equiv \prod_{i=1}^{2m}(x - {r}_i) \bmod p. 
\end{equation}
\begin{enumerate}
 \item[$(1)$] 
 Suppose that $F$ contains a quadratic subfield $M_2$ and that
 the coefficient of $x^{m-1}$ of  the monic minimal polynomial $g_m(x)$ of  $\alpha$ over $M_2$
  be a rational integer $a$.    
 Then, for the decomposition $pO_{M_2}={\mathfrak p}_1{\mathfrak p}_2$ to the product of prime ideals ${\mathfrak p}_i$ of $M_2$,
we  can renumber the roots ${r}_i$ so that
\begin{equation}\label{eq20}
 g_m(x)\equiv\prod_{i=1}^m(x-{r}_i)\bmod {\mathfrak p}_1,\quad
g_m(x)\equiv\prod_{i=m+1}^{2m}(x-{r}_i)\bmod {\mathfrak p}_2,
\end{equation}
and we have the linear relation
\begin{equation}\label{eq21}
 {r}_1+ \dots+ {r}_m \equiv {r}_{m+1}+ \dots+ {r}_{2m} \equiv -a\bmod p.
\end{equation}
Moreover, we have $f(x)= x^{2m} +2ax^{2m-1}+\dots$.
\item[$(2)$]
Suppose that $F$ contains  a subfield $M_{m}$ of  degree $m$ 
and that  the discriminant of the monic minimal quadratic polynomial $g_2(x)$ of $ \alpha$ over $M_m$ is a  rational integer $D$. 
Then, we can renumber the roots ${r}_i$ so that
 we have
\begin{equation}\label{eq22}
 g_2(x)\equiv(x-{r}_i)(x-{r}_{i+m})\bmod {\mathfrak p}_i\quad
(i=1,\dots,m)
\end{equation}
for the decomposition
       $pO_{M_m}={\mathfrak p}_1\dots{\mathfrak p}_m$ to the product of prime ideals, 
and we have the quadratic relation
\begin{equation}\label{eq23}
 ({r}_i - {r}_{i+m})^2 \equiv D \bmod p\quad(i=1,\dots,m).
 \end{equation}
Moreover, $F$ contains a quadratic field $\mathbb{Q}(\sqrt{D}).$
\end{enumerate}
\end{proposition}
\begin{proof}
We number the roots $r_i$ of $f(x)\equiv0\bmod p$ and prime ideals
${\mathfrak P}_i$ of $F$ lying above $p$ so that $\alpha \equiv r_i\bmod {\mathfrak P}_i$.
Let us prove case (1) above. 
First, we note that the degree of $g_m(x)\in O_{M_2}[x]$ is $m$. 
We may assume that $pO_{M_2}={\mathfrak p}_1{\mathfrak p}_2$ and
${\mathfrak p}_1O_F={\mathfrak P}_1\dots{\mathfrak P}_m$ and  ${\mathfrak p}_2
O_F={\mathfrak P}_{m+1}\dots{\mathfrak P}_{2m}$,
which imply ${\mathfrak P}_i\cap M_2={\mathfrak p}_1\,(i=1,\dots,m)$ and
${\mathfrak P}_i\cap M_2={\mathfrak p}_2\,(i=m+1,\dots,2m)$.
The assumptions $g_m(\alpha)= 0$ and $\alpha\equiv{r}_i\bmod{\mathfrak P}_i$ imply
 $g_m({r}_i)\equiv0 \bmod{\mathfrak P}_i$; 
Hence,  
$$ 
g_m({r}_i)\in{\mathfrak P}_i \cap
M_2=
\begin{cases}
{\mathfrak p}_1&(i=1,\dots,m),\\
{\mathfrak p}_2&(i=m+1,\dots,2m),
\end{cases}$$
which concludes \eqref{eq20}.
Therefore, the definition of $a$ implies
 $a+\sum_{i=1}^m r_i \in{\mathfrak p}_1 \cap{\mathbb{Z}}=p{\mathbb{Z}}$ and 
$a+\sum_{i=m+1}^{2m} {r}_i \in{\mathfrak p}_2\cap{\mathbb{Z}} =p{\mathbb{Z}}$
by
$a,{r}_i\in{\mathbb{Z}}$; hence, we get \eqref{eq21}.
Equations $a_{2m-1}+\sum_{i=1}^{2m} {r}_i\equiv 0\bmod p$
and \eqref{eq21}
imply $a_{2m-1}\equiv 2a \bmod p$; hence, $a_{2m-1}=2a$,
since $p$ is sufficiently large.
Thus, we have $f(x)=x^{2m} + 2ax^{2m-1}+\dots$.

\noindent
Next, let us prove case (2) above.
Put $g_2(x) =x^2+Ax+B \,(A,B\in O_{M_m})$.
The assumption $g_2(\alpha ) = 0$ implies
$g_2({r}_i)\equiv0\bmod{\mathfrak P}_i$, that is, $g_2(r_i)\in{\mathfrak P}_i$ $(i=1,\dots,2m)$.
By renumbering, we may assume that
\begin{equation*}
 pO_{M_m}={\mathfrak p}_1\dots{\mathfrak p}_m,\quad{\mathfrak p}_iO_F={\mathfrak P}_i{\mathfrak P}_{i+m}\quad(i=1,\dots,m).
\end{equation*}
Then we have
\begin{equation*}
 g_2({r}_i)\in{\mathfrak P}_i\cap M_m={\mathfrak p}_i,\quad
g_2({r}_{i+m}) \in{\mathfrak P}_{i+m}\cap M_m={\mathfrak p}_i\quad(i=1,\dots,m),
\end{equation*}
that is \eqref{eq22}.
Therefore, we have
\begin{equation*}
 D\equiv ({r}_i+{r}_{i+m})^2 - 4{r}_i{r}_{i+m}
\equiv ({r}_i-{r}_{i+m})^2 \bmod{\mathfrak p}_i\quad(i=1,\dots,m).
\end{equation*}
Since $D$ and ${r}_i$ are rational integers and ${\mathfrak p}_i\cap{\mathbb{Z}}=p{\mathbb{Z}}$, 
we have \eqref{eq23}.
Since the difference  $\sqrt{D}$  of $\alpha$ and its conjugate over
$M_m$ is in $F$
and $D$ is a rational integer, $F$ contains a quadratic field $\mathbb{Q}(\sqrt{D})$.
\end{proof}

A sufficient condition for the assumption in (1) is as follows.
\begin{proposition}
If $f(x)= g(h(x))$ holds for a polynomial $g(x)$ of degree $2$ and a
polynomial $h(x)$ of degree $m\,(>1)$, 
then the assumption in $(1)$ of Proposition\,\ref{prop5} is satisfied.
\end{proposition}
\begin{proof}
Let $\alpha$ be a root of $f(x)$.
Substituting $\beta := h(\alpha)$ and $M_2:=\mathbb{Q}(\beta)$,
we have $g(\beta)=f(\alpha)=0$; hence, $M_2$ is a quadratic field and 
 $g(x)$ is $(x-\beta)(x-\overline{\beta})$ for a conjugate $\overline{\beta}\in M_2$ of
$\beta$ over $\mathbb{Q}$.
Then, $g_m(x):=h(x)-\beta$ satisfies $f(x) = g_m(x)(h(x)-\overline{\beta})$, $g_m(\alpha)=0$, and the second leading coefficient of $g_m(x)$, which is equal to  that of $h(x)$, is rational.
If $g_m(x)$ is reducible over $M_2$, there is a decomposition
$g_m(x)=k_1(x)k_2(x)$ with $k_i(x) \in M_2[x]$ and $\deg k_i>1$.
Thus, $f(x)=(h(x)-\beta)(h(x)-\overline{\beta})=g_m(x)\overline{g_m(x)}$  is divisible by
a polynomial $k_i(x)\overline{k_i(x)}\in\mathbb{Q}[x]$,
which contradicts the irreducibility of $f(x)$.
\end{proof}

Let us make a few comments on the relations between \eqref{eq20} and the distribution.
\begin{lemma}
Keep the case $(1)$ in Proposition\,\ref{prop5} and assume that $0\le r_i<p$ $(1\le i \le 2m)$.
Substituting $C_p(g,{\mathfrak p}_1) :=(a+\sum_{i=1}^m{r}_i)/p\linebreak[3]\in{\mathbb{Z}},
\,C_p(g,{\mathfrak p}_2) :=(a+\sum_{i=m+1}^{2m}{r}_i)/p\in{\mathbb{Z}}$,
we have $C_p(f) =C_p(g,{\mathfrak p}_1)+C_p(g,{\mathfrak p}_2) $ and
\begin{equation}\label{eq24}
 C_p(g,{\mathfrak p}_1)=\lceil({r}_1+\dots+{r}_{m-1})/p\rceil ,\,
 C_p(g,{\mathfrak p}_2)=\lceil({r}_{m+1}+\dots+{r}_{2m-1})/p\rceil 
\end{equation} 
except finitely many primes $p$.
 \end{lemma}
\begin{proof}
The definition \eqref{eq2} of $C_p(f)$ implies $C_p(f)p = 2a+\sum_{i=1}^{2m}
{r}_i=(a+\sum_{i=1}^m{r}_i)+(a+\sum_{i=m+1}^{2m}{r}_i)$, i.e., $C_p(f) =C_p(g,{\mathfrak p}_1)+C_p(g,{\mathfrak p}_2) $.
Substituting $k =\lceil({r}_1+\dots+{r}_{m-1})/p\rceil  $,
we have $({r}_1+\dots+{r}_{m-1})/p \le k <( {r}_1+\dots+
{r}_{m-1})/p+1 $.
Hence, 
$C_p(g,{\mathfrak p}_1) - ({r}_m+a)/p \le k< C_p(g,{\mathfrak p}_1) -
({r}_m+a)/p+1$, and so
$$ - ({r}_m+a)/p \le k -  C_p(g,{\mathfrak p}_1) <-({r}_m+a)/p+1. $$
If $k -  C_p(g,{\mathfrak p}_1) \le -1$ holds, then we have $ - ({r}_m+a)/p\le-1 $;
hence, $1\le p-{r}_m\le a$.
If this inequality holds for infinitely many primes, there is an integer
$r$ between $1$ and $a$ such that $p-{r}_m= r$ for infinitely many
primes,
which implies $f(-r)\equiv f({r}_m)\equiv 0 \bmod p$, hence a
contradiction $f(-r)=0.$
Thus, $k -  C_p(g,{\mathfrak p}_1) \ge0$ holds.
Next, suppose that $k -  C_p(g,{\mathfrak p}_1) \ge 1$ holds for infinitely many
primes.
Then, we have $ ({r}_m+a)/p<0 $,
and hence, $0\le {r}_m< -a$ for infinitely many primes, which is also a
contradiction similar to the above.
Hence, we have $k -  C_p(g,{\mathfrak p}_1) =0$.
Another equality is similarly proved.
\end{proof}
Keeping and continuing the above,
\eqref{eq8} implies
\begin{equation*}
 \lim_{X\to\infty}F_k(f,S,X)
=\lim_{X\to\infty}\displaystyle\frac{\#\{p\in Spl(f,X)\mid C_p(g,{\mathfrak p}_1)+C_p(g,{\mathfrak p}_2)=k\}}
{\#Spl(f,X)}
\end{equation*}
 for any subset $S$ with $\#S=n-1$.
We note that there are $2(m!)^2$ ways of choosing  points
$({r}_1/p,\dots,{r}_{m-1}/p,
{r}_{m+1}/p,\dots,{r}_{2m-1}/p)\in[0,1)^{2(m-1)}$
by two ways for ${\mathfrak p}_1,{\mathfrak p}_2$  and $m!$ ways of choosing
 ${r}_1,\dots,{r}_{m-1}$ (resp. ${r}_{m+1},\dots,{r}_{2m-1}$ ) from 
${r}_1,\dots,{r}_m$ (resp. ${r}_{m+1},\dots,{r}_{2m}$).
{\it If, therefore, a sequence of  all  $2(m!)^2$ points
$({r}_1/p,\dots,{r}_{m-1}/p,{r}_{m+1}/p,\dots,{r}_{2m-1}/p)
\in[0,1)^{2(m-1)}$ for every prime $p\in Spl(f)$
distributes uniformly when $p\to\infty$},
then by \eqref{eq24} we have
\begin{align*}
&\lim_{X\to\infty}F_k(f,S,X)
\\
=\hspace{1mm}&vol(\{(x_1,\dots,x_{2(m-1)})\in[0,1)^{2(m-1)}\mid
\lceil \sum_{l=1}^{m-1}x_l\rceil+\lceil \sum_{l=m}^{2(m-1)}x_l\rceil=k
\})
\\
=&\sum_{i+j=k}vol\{(x_1,\dots,x_{  m-1} )\in[0,1)^{m-1}\mid\lceil \sum_{l=1}^{m-1}x_l\rceil=i\}\times
\\
&\hspace{2cm}vol\{(x_m,\dots,x_{2(m-1)})\in[0,1)^{m-1}\mid\lceil \sum_{l=m}^{2(m-1)}x_l\rceil=j\}.
\end{align*}
The volume $vol\{(x_1,\dots,x_{  m-1} )\in[0,1)^{m-1}\mid\lceil
\sum_{l=1}^{m-1}x_l\rceil=i\}$ is given by Eulerian numbers as above.
In the case of $m=3$ for now, by
$$
vol\{(x_1,x_2)\in[0,1)^2\mid\lceil x_1+x_2\rceil=i\}=
\left\{
\begin{array}{cll}
 0&\text{ if } & i \le0, \\
 1/2&\text{ if } & i =1,2, \\
 0&\text{ if } & i \ge2,
\end{array}
\right.
$$,
we have
$$
\lim_{X\to\infty}F_k(f,S,X)=
\left\{
\begin{array}{cl}
 1/4&\text{ if }  k =2,4, \\
 1/2&\text{ if }  k=3, \\
 0&\text{ otherwise. }  
\end{array}
\right.
$$
This elucidates $Pr(f,S)=[0,1,2,1,0]/4$ at $\#S=5$ in the cases of
Examples 2 and 3.

In the case of $\deg f = 4$, the assumption in (1) of Proposition\,\ref{prop5} and that of being decomposable are equivalent as follows. 
\begin{proposition}
Let ${M_2}={\mathbb{Q}}(\sqrt{D})$  $(D\in {\mathbb{Q}})$ be a quadratic field and $f(x)\in {\mathbb{Q}}[x]$ be a polynomial of degree $\ell+2$. Suppose that $f(x)=g(x)h(x)$ with $g(x)=x^2 +ax+b_1+b_2\sqrt{d},h(x)\in {M_2}[x]$ with  $a,b_1,b_2\in {\mathbb{Q}}$. Then,  $\ell=2$ and $f(x)$ is equal to $(x^2+ax+b_1)^2  - b_2^2\,d$, in particular, decomposable, which implies that in  $(1)$ of Proposition\,\ref{prop5}, the polynomial $f$ is decomposable if $\deg f=4$.
\end{proposition}
\begin{proof}
We note that the irreducibility of a polynomial $f$ implies $b_2\not=0$.
Write $h(x)= x^\ell+h_1(x) +\sqrt{d}\,h_2(x)$ $(h_1,h_2\in {\mathbb{Q}}[x])$;
then, we have 
\begin{align*}
 f(x)&=(x^2 +ax+b_1+b_2\sqrt{d})( x^\ell+h_1(x)
+\sqrt{d}\,h_2(x)   )
\\
&=(x^2+ax+b_1)(x^\ell+h_1(x)) +b_2dh_2(x)
\\
&+\sqrt{d}[ b_2(x^\ell+h_1(x)) +(x^2+ax+b_1)h_2(x)  ]\in {\mathbb{Q}}[x].
\end{align*}
Thus, we have $ b_2(x^\ell+h_1(x)) +(x^2+ax+b_1)h_2(x)  =0 $, and hence,
$x^\ell+h_1(x)=-b_2^{-1} (x^2+ax+b_1)h_2(x) $,
which implies $h(x) = -b_2^{-1} (x^2+ax+b_1-b_2\sqrt{d}  )h_2(x)$.
Thus, we have $f(x)= -b_2^{-1} ((x^2+ax+b_1)^2  - b_2^2\,d)h_2(x)$.
Since $f(x)$ is irreducible and monic, we have  $f(x)= (x^2+ax+b_1)^2  -
b_2^2\,d$.
\end{proof}

\end{document}